\documentclass[reqn,12pt]{amsart}
\usepackage{amsmath, amssymb, amsthm}
\usepackage{amsfonts}
\usepackage[mathcal,mathscr]{eucal}
\usepackage{graphicx, xcolor}
\usepackage[colorlinks]{hyperref}
\newtheorem{theorem}{Theorem}[section]

\newtheorem{lemma}[theorem]{Lemma}

\newtheorem{remark}[theorem]{Remark}
\newtheorem*{conjecture}{Conjecture}

\DeclareMathOperator{\Aut}{Aut}
\DeclareMathOperator{\Sz}{Sz}
\DeclareMathOperator{\Sol}{Sol}
\DeclareMathOperator{\Syl}{Syl}
\DeclareMathOperator{\GL}{GL}
\DeclareMathOperator{\PGL}{PGL}
\DeclareMathOperator{\SL}{SL}
\DeclareMathOperator{\PSL}{PSL}
\DeclareMathOperator{\GU}{GU}

\DeclareMathOperator{\PSU}{PSU}
\DeclareMathOperator{\PSp}{PSp}
\DeclareMathOperator{\Po}{P}
\DeclareMathOperator{\Li}{L}
\DeclareMathOperator{\Un}{U}

\DeclareMathOperator{\GAP}{GAP}

\newcommand{\sub}{\leqslant}

\newcommand{\Z}{{\mathcal Z}}
\newcommand{\C}{{\mathcal C}}
\newcommand{\N}{{\mathcal N}}
\newcommand{\h}{{\mathcal H}}

\newcommand{\cyc}[1]{\langle #1\rangle}
\newcommand{\Fix}{{\rm Fix}}

\newcommand{\nor}{\trianglelefteq}

\newcommand{\bk}{\backslash}

\begin{document}
\begin{abstract}
Let $G$ be a finite group and $x$ be an element of $G$. Define $\Sol_G(x)$ as the set of all $y \in G$ such that $\cyc{x,y}$  is soluble. We provide an equivalent condition for the normalizer-solubilizer conjecture, namely $|\N_G(\cyc{x})| \mid |\Sol_G(x)|$, where $\N_G(\cyc{x})$ is the normalizer of $\cyc{x}$. Furthermore, we demonstrate that the conjecture holds in the special case where $\N_G(\cyc{x})$ is a Frobenius group with kernel $\C_G(x)$, the centralizer of $x$ and $|\N_G(\cyc{x}): \C_G(x)|$ is of prime order. Finally, we will classify all finite simple groups $G$  that contain an element $x$ for which $\Sol_G(x)$ is a maximal subgroup of order $pq$, where $p$ and $q$ are prime numbers.
\end{abstract}

\title[On the Normalizer-Solubilizer Conjecture]{On the Normalizer-Solubilizer Conjecture}
\author[H. Mousavi]{Hamid Mousavi}
\address{Department of Mathematics, Statistics, and Computer Sciences, University of  Tabriz, Tabriz, Iran.}
\email{hmousavi@tabrizu.ac.ir}

\maketitle
\section{Introduction}
Let $G$ be a finite group. For an element $x\in G$, the {\em solubilizer} of $x$ in $G$ is define as
$$
\Sol_G (x):= \{ g \in G ~|~ \langle x, g \rangle \ \mbox{is soluble} \}.
$$

By a result of  John G. Thompson~\cite[Corollary 2]{Thompson} a finite group $G$ is soluble if and only if every two-generated subgroup of $G$ is soluble.Therefore, $G$ is soluble if and only if $\Sol_G(x)=G$ for all $x\in G$. By~\cite[Theorem 1.1]{GKPS}, $x\in R(G)$, the soluble radical of $G$ if and only if $\cyc{x,g}$ is soluble for any $g\in G$. So,  $\Sol_G(x)=G$ if and only if $x\in R(G)$.

In 2013, Hai-Reuven~\cite{Hai}, prove that, for every $x\in G$,  $|\C_G(x)|$ divides $|\Sol_G(x)|$, where $\C_G(x)$ is the centralizer of $x$ in $G$. In 2023, Mousavi and et al.~\cite{Mousavi}, propose the following conjecture:

\begin{conjecture}[Normalizer-Solubilizer Conjecture]
Let $G$ be a finite insoluble group. Then for any $x\in G$,
$|\N_G(\langle x \rangle)|\mid|\Sol_G(x)|$, where $\N_G(\cyc{ x})$ is the normalizer of $x$ in $G$.
\end{conjecture}

The authors of \cite{Akbari1} confirm the validity of this conjecture for minimal simple groups and all groups of orders up to $2000000$.
The purpose of this paper is firstly to present an alternative condition that is equivalent to the normalization conjecture, and second is to prove the correctness of the normalizer-solubilizer conjecture when $\N_G(\cyc{x})$ is a Frobenius group with kernel $\C_G(x)$.  Additionally, as an application of these results, the paper classifies the finite simple groups $G$ for which, for some $x\in G$, $\Sol_G(x)$ is a maximal subgroup of order $pq$, where $p$ and $q$ are distinct primes.

All notations used are standard. We denote the normalizer of the subgroup generated by $x$ as $\N_G(\cyc{x})$, the centralizer of the element $x$ as $\C_G(x)$, soluble radical of $G$ as $R(G)$, the set of Sylow $p$-subgroups of $G$ as $\Syl_p(G)$, where $p$ is a prime number and set of all conjugates of $x$ by elements of $G$ as $x^G$. Additionally, $C_n$ represents a cyclic group of order $n$, $D_{2n}$ denotes the dihedral group of order $2n$. In this paper, all finite simple groups are discussed using standard notation.

\section{Preliminaries}
The following lemmas will aid us in proving the main results.

\begin{lemma}\label{sol}\cite{Akbari, Hai}
Let $G$ be a finite group and $N\nor G$. Then, the following statements hold for any $x\in G$.
\begin{itemize}
\item[(i)]  For every $x\in G$, $|\mathcal{C}_{G}(x)|$ divides $|\Sol_G(x)|$. So $|x|$ divides $|\Sol_G (x)|$.
\item[(ii)] If $\cyc{x}=\cyc{y}$, then $\Sol_G(x)=\Sol_G(y)$.
\item[(iii)] For any $g\in G$, $\Sol_G (x^g)=\Sol_G (x)^g$.
 \item[(iv)]  If $N$ is soluble, then $|\Sol_G (x)|$ is divisible by $|N|$. In particular, $|\Sol_G (x)|$ is divisible by $|R(G)|$. Furthermore,
 $\Sol_{\frac{G}{N}} (xN)=\frac{\Sol_G (x)}{N}=\{xN\,| x\in\Sol_G(x)\}$.
 \item[(v)] $|\Sol_G (x)|=p$ is a prime number if and only if $G$ is of order $p$.
\item[(vi)] If the elements of $\Sol_G(x)$ for some $x\in G$
commute pairwise, then $G$ is abelian.
  \end{itemize}
 \end{lemma}
\begin{lemma}\cite[Lemma 3.2]{Mousavi}\label{self}
Finite insoluble groups do not have self-normalizing subgroups of prime order.
\end{lemma}

\begin{lemma}\cite[Lemma 3.3]{Mousavi}\label{exp}
Let $G$ be a finite group and $x\in G$. Then either $\N_G(\cyc{x})=\Sol_G(x)$ or $|\Sol_G(x)|> \ell|x|$, where $\ell=\min \{|\cyc{x}: \cyc{x}\cap\cyc{x^y}|\,\big{|}\, y\not\in\N_G(\cyc{x})\}$.
\end{lemma}
By Lemma~\ref{exp}, if $|x|=p$ for some prime number $p$, then either $\N_G(\cyc{x})=\Sol_G(x)$ or $|\Sol_G(x)|> p^2$.

\begin{theorem}\cite[Theorem 4.1]{Wilson}\label{SZ}
If $q = 2^{2n+1}$, $n > 1$, then the maximal subgroups of  the Suzuki group, $\Sz(q)$,  are (up to conjugacy)
\begin{itemize}
\item[(i)] $E_q\cdot E_q\cdot C_{q-1}$, where $E_q$ is elementary abelian $2$-group of order $q$;
\item[(ii)] $D_{2(q-1)}$;
\item[(iii)] $C_{q+1+\sqrt{2q}}:C_4$;
\item[(iv)] $C_{q+1-\sqrt{2q}}:C_4$;
\item[(v)] $\Sz(q_0)$, where $q = q_0^r$, $r$ is prime, and $q_0 > 2$.
\end{itemize}
\end{theorem}

A divisor $r$ of $q^d-1$ is said to be a primitive divisor if $r\nmid q^i-1$ for $i<d$, and the largest primitive divisor of $q^d-1$ is called the primitive part of  $q^d-1$. If the divisor $r$ is prime, it is called a primitive prime or Zsigmondy prime, denoted by $z_{q,d}$ as mentioned in~\cite{BHD}.
For  any $(q,d)$, the primitive prime divisor of $q^d-1$ is not unique, for example, the primitive prime divisors of $17^6-1$ are $7$, $13$ and $307$. 


\begin{theorem}\cite[Theorem 1.13.1]{BHD}{\rm(Zsigmondy~\cite{Zsi})}
  Let $q\geq 2$ be a prime power and $d\geq 3$, with $(q, d) \neq (2, 6)$. Then there exists at least one prime $r$ such that $r$ divides $q^d - 1$ but does not divide $q^i-1$ for all $i<d$.
\end{theorem}

\begin{lemma}\cite[Lemma 1.13.3]{BHD}\label{Primitive}
\begin{itemize}
\item[(i)] Let $q \geq 2$ be a prime power and $n\geq 3$. Then $q^d + 1$ is
divisible by $z_{q,2d}$ if and only if $(q, d)\neq (2, 3)$.
\item[(ii)] If $z_{q,d}$ divides $q^m - 1$ then $d$ divides $m$.
\item[(iii)] The prime $z_{q,d}\equiv 1 \pmod{d}$, so in particular, $z_{q,d} > d$.
\end{itemize}
\end{lemma}

\section{Main Results}
In this section, we present our main results, including the compatibility condition for the normalizer-solubilizer conjecture (Theorem~\ref{T-3.4}). Additionally, we discuss the connectedness of the normalizer-solubilizer conjecture in the special case where the normalizer of an element $x$ in group $G$ is a Frobenius group, and its kernel $\C_G(x)$ has prime index (Theorem~\ref{frob}).

\begin{lemma}\label{L-3.1}
Let $G$ be a finite group and $x\in G$. Then $\N_G(\cyc{x})$ acts by conjugation on $\Sol_G(x)$.
\end{lemma}
\begin{proof}
Assume that $y\in\Sol_G(x)$ and $t\in \N_G(\cyc{x})$. Then for some $i$ coprime to $|x|$,
\[\cyc{x,y^t}=\cyc{x^{t^{-1}},y}^t=\cyc{x^i,y}^t=\cyc{x,y}^t.\] 
 Hence $\cyc{x,y^t}$ is soluble and so $y^t\in\Sol_G(x)$. 
\end{proof}

We consider the following notations for this section. Let $G$ be a finite group and $x\in G$. We set $\N_x=\N_G(\cyc{x})$ and $\C_x=\C_G(x)$. By Lemma~\ref{L-3.1}, for any subgroup $\h$ of $G$ such that $\C_x\sub\h\sub\N_x$, $\h$ acts by conjugation on $\Sol_G(x)$. The number of orbits resulting from this action is denoted by $\ell_{\h}$. So for $\h=\C_x$ or $\N_x$, we have $\ell_{\C_x}$ and $\ell_{N_x}$.

As $\h\nor\N_x$, then, $\h\bk\{1\}=\cup_{i=1}^{\ell} g_i^{\N_x}$ for some $g_i\in\h$ and $\ell\in\mathbb{N}$. Also for any $g\in \h$, we set $$n(g,x):=|\Sol_{\C_G(g)}(x)|/|\N_{\C_G(g)}(\cyc{x})|.$$
Note that $n(1,x)=|\Sol_G(x)|/|\N_x|$. Now, for any $u\in \h$, as
 $\N_{\C_G(g)}(\cyc{x})=\N_x\cap\C_G(g)$, then 
\[|\N_x\cap\C_G(g^u)|=|(\N_x\cap\C_G(g))^u|=|\N_x\cap\C_G(g)|.\]
Also,
\[|\Sol_G(x)\cap\C_G(g^u)|=|\Sol_G(x^{u^{-1}})\cap\C_G(g)|=|\Sol_G(x)\cap\C_G(g)|.\]
The last equality follows from Lemma~\ref{sol}-(ii), since $\cyc{x}=\cyc{x^u}$.  Therefore, for any $u\in \h$, $n(g,x)=n(g^u,x)$.

\begin{lemma}\label{L-3.2}
Let $G$ be a finite group. By the notation above, for any $x\in G$, we have
\[|\h|\ell_{\h}=|\Sol_G(x)|+|\N_x|\sum_{i=1}^{\ell} n(g_i, x).\]
\end{lemma}
\begin{proof}
We note that
$|\h|\ell_{\h}=\sum_{g\in\h}|\Fix_Y(g)|$, where $Y=\Sol_G(x)$, 
\allowdisplaybreaks{\begin{align*}
|\h|\ell_{\h} &=\sum_{g\in\h}|\Fix_Y(g)|\\
&=|\Sol_G(x)|+\sum_{1\neq g\in\h}|\Fix_Y(g)|\\
&=|\Sol_G(x)|+\sum_{1\neq g\in\h}|\C_G(g)\cap\Sol_G(x)|\\
 &=|\Sol_G(x)|+\sum_{1\neq g\in\h}|\Sol_{\C_G(g)}(x)|\\
 &=|\Sol_G(x)|+\sum_{1\neq g\in \h}n(g,x)|\N_{\C_G(g)}(\cyc{x})|\\
&=|\Sol_G(x)|+\sum_{1\neq g\in \h}n(g,x)|\C_{\N_x}(g)|\\
&=|\Sol_G(x)|+\sum_{i=1}^{\ell} \sum_{g\in g_i^{\N_x}} n(g, x)|\C_{\N_x}(g)|\\
&=|\Sol_G(x)|+\sum_{i=1}^{\ell} n(g_i, x)|\C_{\N_x}(g_i)||g_i^{\N_x}|\\
&=|\Sol_G(x)|+|\N_x|\sum_{i=1}^{\ell} n(g_i, x).
\end{align*}}
\end{proof}

By Lemma~\ref{L-3.2} we can provide a formula for calculating the number of orbits resulting from the conjugation action of $\h$ on $\Sol_G(x)$ (especially for $\h=\C_x$ or $\N_x$). We note that, if $G$ is soluble, then for any $x\in G$, and any $H\sub G$, $\Sol_H(x)=H$. Hence, for any $1\neq g\in \C_x$, $$n(g,x)=|\Sol_{\C_G(g)}(x)|/|\N_{\C_G(g)}(\cyc{x})|=|\C_G(g):\N_x\cap\C_G(g)|.$$ According to Lemma~\ref{L-3.2}, we have:
\begin{equation}\label{E-1}
\ell_{\h}=|G:\h|+|\N_x:\h|\sum_{i=1}^{\ell} |\C_G(g_i):\N_x\cap\C_G(g_i)|,
\end{equation}
where $\h\bk\{1\}=\cup_{i=1}^{\ell} g_i^{\N_x}$.

Instead of the solubility assumption, a weaker assumption can be considered. For example, let $G=KH$ be a Frobenius group with kernel $K$ and complement $H$. Suppose that $1\neq x\in K$. Since $K\sub R(G)$, so $\Sol_G(x)=G$ and, for any $g\in\C_x\bk\{1\}$, as $\C_G(g)\sub K$, $\Sol_{\C_G(g)}(x)=\C_G(g)$. Hence, by Lemma~\ref{L-3.2}, $\ell_{\C_x}$ is given by Equation~(\ref{E-1}) with substituting $\C_x$ for $\h$. Additionally, if $K$ is abelian, then $\C_x=\C_G(g)=K\sub\N_x$, so $$\ell_{\C_x}=|H|+\ell|\N_x: K|,$$ where $\ell+1$ is the number of orbits resulting from the action of $\N_x/C_x$ on $K$, for $\C_x$ acts trivially on $K$.

\begin{remark}\label{R-3.3}
We note in Lemma~\ref{L-3.2},
\[\ell_{\h}=|\Sol_G(x)|/|\h|+|\N_x: \h|\sum_{i=1}^{\ell} n(g_i, x).\]
So, if $\sum_{i=1}^{\ell} n(g_i, x)$ is an integer, where $g_i\in\h$, then $|\h|\mid |\Sol_G(x)|$, in particular when $\h=\N_x$. On the other hand,
\[\frac{|\C_x|}{|\N_x|}\ell_{\C_x}=\frac{|\Sol_G(x)|}{|\N_x|}+\sum_{i=1}^{\ell} n(g_i, x).\]
Now, if normalizer conjecture is true, this means that, for any group $G$ and for any $x\in G$, $|\N_x|\mid |\Sol_G(x)|$, then for any $1\leq i\leq \ell$, $n(g_i,x)$ is an integer, where $g_i\in\C_x\bk\{1\}$. So $\frac{|\C_x|}{|\N_x|}\ell_{\C_x}$ is an integer.
\end{remark}

Let $G\cong A_5$, the alternating group of degree 5. Using GAP~\cite{GAP}, we compute the solubilizer of the conjugacy class representative of elements in $G$,  $|\C_x|$, $|\N_x|$  and $\ell_{\C_x}$.

\begin{table}[h]
\begin{tabular}{|c|c|c|c|c|}\hline
$x$ &(\,)& (1\, 2)	& (1\,2\,3)& (1\,2\,3\,4\,5)\\ 
\hline
$|\Sol_G(x)|$ & 60 & 36& 24& 10\\ 
\hline
$|\N_x|$ & 60 & 4& 6& 10\\ 
\hline
Structure of $\N_x$ & $A_5$ & $C_2\times C_2$ & $S_3$ & $D_{10}$\\
\hline
$|\C_x|$ & 60& 4& 3& 5\\ 
\hline
$\ell_{\C_x}$ & 1&12& 10&6\\ 
\hline
$\frac{|\C_x|}{|\N_x|}\ell_{\C_x}$ &  1& 12& 5& 3\\
\hline
\end{tabular}
\caption{Analysis of the group $A_5$ }\label{T-1}
\end{table}
In the group $A_5$, cycles of length $5$ form two distinct conjugacy classes, represented by $(1\, 2\, 3\, 4\, 5)$ and $(2\, 1\, 3\, 4\, 5)=(1\, 2\, 3\, 4\, 5)^{(1\, 2)}$. As $\Sol_G(x^g)=(\Sol_G(x))^g$, $C_{x^g}=\C_x^g$ and $\N_{x^g}=\N_x^g$, so $\ell_{\C_{x^g}}=\ell_{\C_x}$. Therefore, we have the same result as $(1\,2\,3\,4\,5)$ for $x=(2\, 1\, 3\, 4\, 5)$.

By  Table~\ref{T-1}, the value of $\frac{|\C_x|}{|\N_x|}\ell_{\C_x}$ is always an integer for group $A_5$. 
\begin{theorem}\label{T-3.4}
The following statements are equivalent.
\begin{itemize}
\item[(i)] For any finite group $G$ and any $x\in G$, $|\N_x|\mid|\Sol_G(x)|$.
\item[(ii)] For any finite group $G$ and any $x\in G$, $\frac{|\C_x|}{|\N_x|}\ell_{\C_x}$ is an integer.
\end{itemize}
\end{theorem}

\begin{proof}
Assume that (i) holds. Let $G$ be a finite group and $x\in G$. Then for any $g\in\C_x$, by (i), $|\N_{\C_G(g)}(\cyc{x})|$ divides $|\Sol_{\C_G(g)}(x)|$. So $n(g,x)$ is an integer  (note that $n(1,x)=|\Sol_G(x)|/|\N_x|$). Now, by Lemma~\ref{L-3.2},
\[\frac{|\C_x|}{|\N_x|}\ell_{\C_x}=\frac{|\Sol_G(x)|}{|\N_x|}+\sum_{i=1}^{\ell}n(g_i,x),\]
where $\C_x\bk\{1\}=\cup_{i=1}^{\ell}g_i^{\N_x}$. Therefore $\frac{|\C_x|}{|\N_x|}\ell_{\C_x}$  is an integer and (ii) holds.

Assume that (ii) holds. We show that for any group $G$ and any $x\in G$, $|\N_x|\mid |\Sol_G(x)|$. Let $G$ be a finite group and $x\in G$. We will use induction on $|G|$. Assume that for any group $H$ with an order less than $|G|$ and for  any $x\in H$, $|\N_H(\cyc{x})|\mid |\Sol_H(x)|$ (we note that this is true for any group of order less that $60$). We set  $\Z := Z(G)$. Suppose that $\Z \neq 1$ and $\bar{G}=G/\Z$.  Since $\N_{\bar{G}}(\cyc{\bar{x}})=\N_x/\Z$ and $\Sol_{\bar{G}}(\bar{x})=\Sol_G(x)/\Z$ then, by induction, we find that $|\N_x/\Z|\mid |\Sol_G(x)/\Z|$, which implies that $|\N_x|\mid |\Sol_G(x)|$.

Now suppose that, $\Z=1$. Then for all $1\neq g\in\C_x$,  $\C_G(g)$ is a proper subgroup of  $G$ contains $x$. By induction,
\[|\N_{\C_G(g)}(\cyc{x})|\mid|\Sol_{\C_G(g)}(x)|=|\C_G(g)\cap\Sol_G(x)|.\]
Therefore, $n(g,x)$ is an integer and by Lemma~\ref{L-3.2}, with $\h=\C_x$ we have
\[\frac{|\Sol_G(x)|}{|\N_x|}=\sum_{i=1}^{\ell}n(g_i,x)-\frac{|\C_x|}{|\N_x|}\ell_{\C_x},\]
is an integer. Therefore (i) holds.
\end{proof}

In the following we show that, when $\N_x$ is a Frobenius group with kernel $\C_x$ for some $x\in G$, then  $|\N_x|\mid|\Sol_G(x)|$, if $|\N_x: C_x|$ is prime. We note that the conjecture always is true for a soluble group $G$, since $\Sol_G(x)=G$ for any $x\in G$. Therefore, the insoluble case is important.

\begin{theorem}\label{frob}
Let $G$ be a finite insoluble group and $1\neq x\in G$. If $\N_x$ is a Frobenius group with kernel $\C_x$ such that $|\N_x/\C_x|$ is a prime number.  Then $|\N_x|\mid|\Sol_G(x)|$.
\end{theorem}
\begin{proof}
Without loss of generality, we can assume that $x\not\in R(G)$. 
Let the result be considered true for every insoluble group of order less than $|G|$ (note that the induction process begins with $A_5$, as shown in Table~\ref{T-1}, where $|x|=3$ or $5$. In these cases, respectively, $\N_G(\cyc{x})\cong S_3$ or $D_{10}$). 

As $\N_x$ is a Frobenius group, $ Z(G)\sub Z(\N_x)=1$. Thus for any $1\neq g\in\N_x$, $\C_G(g)$ is a proper subgroup of $G$. First assume that $1\neq g\in\C_x$. Then $x\in\C_G(g)$, so $\C_x\cap\C_G(g)\neq 1$. If $\C_G(g)$ is soluble, then $ \Sol_{\C_G(g)}(x)=\C_G(g)$. Hence $|\N_{\C_G(g)}(\cyc{x})|\mid|\Sol_{\C_G(g)}(x)|$. Otherwise, since  $|\C_{\C_G(g)}(x)|\mid\Sol_{\C_G(g)}(x)|$ by  Lemma~\ref{sol}-(i),  we can assume that 
$$\N_{\C_G(g)}(\cyc{x})=\N_x\cap\C_G(g)\neq\C_x\cap\C_G(g)=\C_{\C_G(g)}(x).$$ Then $\N_{\C_G(g)}(\cyc{x})$ is Frobenius group with kernel $\C_{\C_G(g)}(x)$. Also,
\[1\neq\frac{\N_{\C_G(g)}(\cyc{x})}{\C_{\C_G(g)}(x)}=\frac{\N_x\cap\C_G(g)}{\C_x\cap\C_G(g)}\cong\frac{(\N_x\cap\C_G(g))\C_x}{\C_x}\hookrightarrow\frac{\N_x}{\C_x}\]
So, $|\N_{\C_G(g)}(\cyc{x}): \C_{\C_G(g)}(x)|$ is prime and by induction
 $|\N_{\C_G(g)}(\cyc{x})|\,\mid\,|\Sol_{\C_G(g)}(x)|$.

Now assume that $g\in\N_x\bk\C_x$. Let $H$ be a complement of  $\C_x$ in $\N_x$. By assumption $H$ is cyclic. We set $H=\cyc{t}$. Then for any  
$g\in\N_x\bk\C_x$, $g\in H^u$ for some $u\in\N_x$, so
 $$|\N_{\C_G(g)}(\cyc{x})|=|\C_{\N_x}(g)|=|H^u|=|H|.$$
Now for any $z\in\Sol_{\C_G(g)}(x)$, as $\cyc{x,z,g}\cong\cyc{x,z}\rtimes\cyc{g}$ is soluble, so $\cyc{x,zg^i}$ is soluble too, for any $i\leq |g|$. Thus $zg^i\in\Sol_{\C_G(g)}(x)$, so $zH^u=z\cyc{g}\subset\Sol_{\C_G(g)}(x)$, for $|\cyc{g}|=|H|$. Hence, the cosets of $H^u$ cover the set $\Sol_{\C_G(g)}(x)$. Therefore, $$|\N_{\C_G(g)}(\cyc{x})|=|H^u|\mid|\Sol_{\C_G(g)}(x)|.$$
Thus, for any $1\neq g\in\N_x$,
  $$n(g,x)=|\Sol_{\C_G(g)}(x)|/|\N_{\C_G(g)}(\cyc{x})|$$ 
   is an integer.
 According to Lemma~\ref{L-3.2} with $\h=\N_x$, we get
\[\ell_{\N_x}=\frac{1}{|\N_x|}|\Sol_G(x)|+\sum_{i=1}^{\ell} n(g_i, x).\]
Therefore,
 $|\N_x|\mid|\Sol_G(x)|$.
\end{proof}

In the \cite[Theorem 3.6]{Mousavi} the authors presented the following theorem.
\begin{theorem}\label{2p}
Let $G$ be a finite insoluble group and for some $x\in G$, $|\Sol_G(x)|=2p$, where $p$ is an odd prime number. Then $G$ is simple and $\N_G(\cyc{x})=\Sol_G(x)$.	
\end{theorem}

In \cite[Remark 3.7]{Mousavi}, it is stated that the same result holds when $|\Sol_G(x)|=pq$ and $|x|=q>p$. Now,  the condition $|x|=q>p$ can be removed. The stated theorem can be strengthened as demonstrated below.

\begin{theorem}\label{pq}
Let $G$ be an insoluble group and $x$ be an element of $G$ such that $|\Sol_G(x)|=pq$, where $p\leq q$ are primes. Then $G$ is simple, $|x|=q>3$ and  $p\mid q-1$. In addition, $\Sol_G(x)=\N_G(\cyc{x})$ is a subgroup of $G$.
\end{theorem}
\begin{proof}
Step 1. $\cyc{x}=\C_x$ is a Sylow subgroup of $G$.

Since $|x|\mid|\Sol_G(x)|$, $|x|\in\{p, q\}$. Also $|\C_x|\mid|\Sol_G(x)|$, if $\C_x=\Sol_G(x)$, then $\Sol_G(x)$ is an abelian subgroup of $G$. So $G$ is abelian by Lemma~\ref{sol}-(vi) and we have a contradiction to the insolubility of  $G$. Hence $\cyc{x}=\C_x$ is a Sylow subgroup of $G$. 

Step 2. $|x|=q$ and $p\mid q-1$.

Let $r\in\{p, q\}$ and $\cyc{x}=R\in\Syl_r(G)$. Now by using Lemma~\ref{self}, we have $\cyc{x}\neq\N_x$. Therefore $\N_x$ is a Frobenius group with kernel $\C_x$.  By Theorem~\ref{frob},  $\N_x=\Sol_G(x)$. Therefore $r=q$, $p\mid q-1$ and $\C_G(Q)=Q$.

Step 3. $G$ is simple and $q>3$. 

Assume that $G$ is not simple and $N$ is a  minimal normal subgroup of $G$. If $q\mid\,|N|$, then $G=N\N_G(Q)$. By the step 2, $|\N_G(Q)|=pq$, where $Q=\cyc{x}$, so $\N_N(Q)=Q$. According to Lemma~\ref{self}, $N$ is soluble. Since $\N_G(Q)$ is also soluble, it follows that $G$ is soluble, which is a contradicts. Therefore $q\nmid|N|$, so
 $Q\cap N={1}$ and $QN$ is a Frobenius group with kernel $N$ (because $Q=\C_G(Q)$). Hence $N$ is nilpotent and $NQ$ is soluble. It implies that $NQ\leq \Sol_G(x)=\N_G(Q)$, so $|N|=p$ and $\Sol_G(x)$ is abelian, which contradicts the insolubility of $G$.

If $q=3$, as $|\N_G(Q): C_G(Q)|=2$, then $G\cong S_3$ or $A_4$ by~\cite[Theorem A]{Stw}, which is a contradiction. Therefore $q>3$.
\end{proof}
\begin{remark}
By Theorem~\ref{pq}, we conclude that for any insoluble group $G$ and $x\in G$, $|\Sol_G(x)|\neq 6$.
\end{remark}

\section{Simple groups with $\Sol_G(x)$ as maximal subgroup of order $pq$}
{\bf Now the question arises:} which of the finite simple groups $G$ has an element $x$ of prime order $q$ such that for some prime $p$, $|\Sol_G(x)|=pq$?

The above question is generally difficult to answer, but if  $\Sol_G(x)$ is a maximal subgroup for some $x\in G$, the above question can be checked.
By Theorem~\ref{pq}, $\Sol_G(x)$ is a  $q$-local subgroups of a finite simple group $G$, where $q$ is odd.

Assume that $M\sub G$ is a maximal subgroup of order $pq$, where $q>p$ is prime. Then, either Sylow $q$-subgroup of $G$ is normal or $M=\N_G(\cyc{x})$, where $\cyc{x}\in\Syl_q(G)$. So, if $G$ is simple, then $M$ is a $q$-local maximal subgroup of $G$. 
To address the question above, we must find the maximal subgroups of order $pq$ of all finite simple groups. 

The local maximal subgroups of all finite simple groups are known and listed in the atlas of groups~\cite{Atlas}, except for the $2$-locals maximal subgroups of the Monster and Baby Monster which are found in~\cite[Theorems 1, 2]{Ivanov} if which are of  characteristic $2$  and~\cite[Theorems A, B]{MeSh1} if which are not of  characteristic $2$. All maximal subgroups of the Monster and Baby Monster have been identified with these classifications.

Let $G$ be a finite non-abelian  simple group. Then $G$ is one of the following:
 \begin{itemize}
 \item[(a)] One of the $26$ sporadic groups:
 \begin{itemize}
  \item[$\bullet$] the five Mathieu groups $M_{11}, M_{12}, M_{22}, M_{23}, M_{24}$;
 \item[$\bullet$] the seven Leech lattice groups $Co_1, Co_2, Co_3$,  $McL$, $HS$,  $Suz$, $J_2$;
 \item[$\bullet$] the three Fischer groups $Fi_{22}, Fi_{23}, Fi'_{24}$;
 \item[$\bullet$] the five Monstrous groups $\mathscr{M}=F_1$, $\mathscr{B}=F_2^+$, $Th= F_3$ $HN=F_5$, $He=F_3$;
 \item[$\bullet$] the six pariahs  $J_1,  J_3,  J_4$, $Ly$, $Ru$, $O'N$.
 \end{itemize}
 \item[(b)] One of the exceptional simple group of Lie type: 
 $^3D_4(q)$, $G_2(q)$, $E_6(q)$, $^2E_6(q)$, $E_7(q)$, $E_8(q)$, $F_4(q)$, where q is a prime power, or
  $^2F_4(2^{2n+1})$,  $^2G_2(3^{2n+2})$, where $n\geq 1$ or the Tits group $^2F_4(2)'$.
\item[(c)] The Suzuki group $\Sz(q)$, where $q=2^{2n+1}$.
 \item[(d)] An alternating group $A_n$, of degree $n \geq 5$.
 \item[(e)] One of the finite classical simple group:
 \begin{itemize}
 \item[Linear:] $\PSL(n,q)$, where $n\geq 2$,  except $\PSL(2,2)$ and $\PSL(2,3)$;
\item[Unitary:] $\PSU(n,q)$, $n\geq 3$, except $\PSU(32)$;
\item[Symplectic:] $\PSp(2n,q)$, $n\geq 2$, except $PSp(4,2)$;
\item[Orthogonal:] $\Po\!\Omega(2n+1,q)$, $n\geq 3$, $q$ odd, $\Po\!\Omega^{+}(2n,q)$, $n\geq 4$, $\Po\!\Omega^{-}(2n,q)$, $n\geq 4$, where $q$ is a prime power.
 \end{itemize}
 For $G\in\{G_2(2), ^2G_2(3)\}$, $G'$ is simple and $|G: G'|=q$, where $q=2$ or $3$.  In fact $G_2(2)'\cong\PSU(3,3)$, $^2G(3)'\cong\PSL(2,8)$.
\end{itemize}

(a) We must only consider $q$-local maximal subgroups, where $q$ is an odd prime, which are fully represented in the atlas of groups~\cite{Atlas}. Upon examining the atlas of finite groups, we find that the only sporadic simple groups with a local maximal subgroup of the structure $C_q: C_p$, where $p$ and $q$ are primes, are as follows:
\begin{itemize}
\item The Mathieu group $M_{23}$, with one conjugacy class of maximal subgroups isomorphic to $C_{23}:C_{11}$.
\item The Baby Monster group $\mathscr{B} = F_2^+$, with one conjugacy class of maximal subgroups isomorphic to $C_{47}:C_{23}$.
\item The Monster group $\mathscr{M} = F_1$, with one conjugacy class of maximal subgroups isomorphic to $C_{59}:C_{29}$.
\end{itemize}

(b)  The list of  all maximal subgroups of  $G_2(q)$, $^2G_2(3^{2n+2})$, $^3D_4(q)$, $F_4(q), q\geq 5$ and $^2F_4(2^{2n+1})$ respectively presented in \cite[Table 4.1 and Theorems 4.2, 4,3, 4.4 and 4.5]{Wilson}. None of them has an order that can be expressed as a product of two prime numbers.  According to \cite[Theorems  2, 3 and Table 2]{CLSS} and~\cite[Theorem 1, Tables 5.1, 5.2]{LSS}, all local maximal subgroups of finite exceptional simple groups $E_6(q)$, $^2E_6(q)$, $E_7(q)$, $E_8(q)$ and $F_4(q)$ are the normalizer of some elementary abelian groups of type $(p, p)$ or $(p, p, p)$. 
 The structures of these maximal subgroups are detailed in \cite[Tables 1, 2]{CLSS}. For the group $^2F_4(2)'$ and $F_4(2)$, we can reefer to \cite[pages 74, 170]{Atlas}. Again, none of them has an order that can be expressed as a product of two prime numbers.

(c) The structure of maximal subgroups of Suzuki groups is presented in~\cite[Theorem 4.1]{Wilson}. The Suzuki group $\Sz(2^{2n+1})$, $n>1$ has a maximal subgroup of order $2q$ if and only if $q=2^{2n+1}-1$ is a Mersenne prime number by Theorem~\ref{SZ}-(ii). This maximal subgroup is isomorphic to $D_{2q}$ which is normalizer of  its Sylow $q$-subgroup.

(d) According to the main theorem of ~\cite{LPS}, the alternating group $A_n$ for $n > 5$ does not possess any local maximal subgroup of order $pq$. Note that $A_5\cong\PSL(2,5)$.

(e) Assume that $G\cong\PSp_{2n}(q)$, where $q$ is odd. The structure of maximal subgroups of $G$ is presented in~\cite[Theorems 3.7, 3.8]{Wilson}. According to these theorems, $G$ does not contain any local maximal subgroup of order that can be expressed as a product of two prime numbers.

Assume that $G$ is a non-soluble orthogonal group. Note that for dimension $n\leq 6$, we have the following isomorphisms~\cite[Proposition 2.9.1]{KlLi}.

\allowdisplaybreaks{$\begin{array}{ll}
\Po\!\Omega_3(q)&\cong\PSL(2,q), (q\geq 4).\\
\Po\!\Omega^{+}_4(q)&\cong\PSL(2,q)\times\PSL(2,q), (q\geq 4).\\
 \Po\!\Omega^{-}_4(q)&\cong\PSL(2,q^2).\\
 \Po\!\Omega_5(q)&\cong\PSp_4(q).\\
 \Po\!\Omega^{+}_6(q)&\cong\PSL^{\pm}(4,q).\\
 \Po\!\Omega^{-}_6(q)&\cong\PSU(4,q).
\end{array}$}

Also $\Po\!\Omega_{2k-1}(q)\cong\PSp_{2k-2}(q)$, when $q$ is power of $2$, (see~\cite[Section 3.4.7 and p. 76]{Wilson}). So, we just consider $G\cong\Po\!\Omega_n(q)$, when $nq$ is odd and $n\geq 7$ or $G\cong\Po\!\Omega^{\pm}_{2k}(q)$, when $k\geq 4$.

Assume that $G \cong \Omega_n(q)$ when $nq$ is odd. The maximal subgroups of $G$ are thoroughly described by several results \cite[Propositions 4.1.6 \& 20, 4.2.14-15, 4.3.17, 4.4.18, 4.5.8, and 4.7.8]{KlLi}, and none of them exhibit the expected structure. Also the structures of maximal subgroups of $\Po\!\Omega^{-}_{2k}(q)$, when $k\geq 4$ are presented in~\cite[Propositions 4.1.6-7 \&  20, 4.2.11 \&  14-16, 4.3.16 \&  18 \&  20, 4.4.7 and 4.5.10]{KlLi} and for $\Po\!\Omega^{+}_{2k}(q)$,~\cite[Propositions 4.1.6-7 \&  20, 4.2.7 \&  11 \&  14-16, 4.3.14 \&  18 \&  20, 4.4.12 \&  14-17, 4.5.10, 4.6.8 and 4.7.5-7]{KlLi}. Again, none of them has the structure we expect.

The structure of maximal subgroups of finite projective special linear of dimension $2$, $\PSL(2,r)$, where $r=p^n$ can be found in~\cite[Theorems 6.25 and 6.26]{Suzuki}. 
Assume that $G\cong\PSL(2,r)$, where $r=p^n\geq  4$, if $p=2$ and 
$r=p^n\geq 5$, if $p$ is odd prime. According to~\cite[Theorems 6.25 and 6.26]{Suzuki}, the group $G$ has three types of conjugacy classes of maximal subgroups, with orders that can be expressed as the product of two prime numbers. Two of these classes exhibit a dihedral structure with an order of $2q$, where $q=r\pm 1$ if $p=2$, $q=(r+1)/2$, $r\neq 7, 9$ and $q=(r-1)/2$, $r\geq 13$ if $p$ is an odd prime. The third class is isomorphic to the group $C_r : C_q$, where $n = 1$, $r=p$ is an odd prime and $q=(p-1)/2$. Note that in each case, $q$ must be a prime number.

The structure of maximal subgroups in finite projective special linear and unitary groups of dimensions greater than 2 is extensively documented in various theorems and tables found in ~\cite[Chapter 4]{KlLi}. However, we will focus on specific cases of these structures based on Lemma~\ref{L-4.1}, which will be presented below.

Assume that $d\geq 3$. We set $G^{\eta}=\GL^{\eta}(d,q )$,  where $\eta=\pm$ and $q=p^n$ is power of a prime $p$. In the notation, we use $G^{+}=\GL(d,q)$ is linear group and $G^{-}=\GU(d,q^2)$ is unitary group. Similarly, we set $L^{\eta}=\SL^{\eta}(d,q)$, $\bar{G}^{\eta}=\PGL^{\eta}(d,q)$ and $\bar{L}^{\eta}=\PSL^{\eta}(d,q)$.

 If $r\nmid q-\eta=|\GL^{\eta}(d,q ) : \SL^{\eta}(d,q)|$, then
 the sets of Sylow $r$-subgroups in $\GL^{\eta}(d,q )$ and $\SL^{\eta}(d,q)$ coincide. In the following we show that for any odd prime $r$, $r\nmid q-\eta$, if $|R|=r$, where $R\in\Syl_r(L^{\eta})$.

\begin{lemma}\label{L-4.1}
Let $r$ be a prime divisor of $|L^{\eta}|$ and $R\in\Syl_r(L^{\eta})$. If $|R|=r$, then $r\nmid q-\eta$.
\end{lemma}
\begin{proof}
Let $\bar{L}^{\eta}$ has a Sylow subgroup of odd prime order $r$. Note
\[|\bar{L}^{\eta}|=\frac{1}{(d,q-\eta)}q^{d(d-1)/2}\prod_{i\geq 2}^d(q^i-{\eta }^i).\]
In the $\Li$ case, $(q-1)^{n-1}\mid |\bar{L}^{+}|$ and in the $\Un$ case $(q+1)^{[(d+1)/2]}\mid |\bar{L}^{-}|$. If $r\mid  q-\eta$, as $d-1\geq 2$, so $[(d+1)/2]\geq 2$. Therefore $r\mid (d, q-\eta)$. Since the $r$-part of
\[\frac{(q-\eta)(q+\eta)}{(d,q-\eta)}\prod_{i\geq 3}^d(q^i-{\eta }^i),\]
 is $r$, thus $r\nmid \frac{q-\eta}{(d,q-\eta)}$ (for $r\nmid q+\eta$) and so $r=d=3$. Now we have
\[|\bar{L}^{\eta}|=\frac{1}{3}q^3(q^2-q)(q^3-\eta^3)=\frac{(q-\eta)}{3}(q+\eta)q^3(q-\eta)(q^2+\eta q+\eta^2),\]
and $3\nmid \frac{q-\eta}{3}$. Since 
\[q^2+\eta q+\eta^2 = (q-\eta)(q+\eta) +\eta(q+2\eta),\]
and $r\mid q-\eta$, $r\mid q+2\eta$, thus $r\mid q^2+\eta q+\eta^2 $ and so we have the contradiction  $r^2\mid |\bar{L}^{\eta}|$. 
\end{proof}

According to Lemma~\ref{L-4.1}, If $\PSL^{\eta}(d,q)$ has a Sylow $r$-subgroup of prime order $r$, then, as $r\nmid \frac{q-\eta}{(d, q-\eta)}$, thus for some $n\geq 3$, $r\mid \frac{q^n-\eta^n}{q-\eta}$ and for all $i < n$ such that $i\mid n$, $r\nmid\frac{q^i-\eta^i}{q-\eta}$. So for all $n< i\leq d$, $n\nmid i$, therefore $d/2< n\leq d$. If $n$ is odd, such prime number $r$ is known as a primitive prime divisor of  $q^n-\eta$ when $n$  is the smallest such that  $r\mid q^n-\eta$.

By \cite[Proposition 4.3.6]{KlLi}, for $m=1$, $\PSL^{\eta}(d,r)$ has a maximal subgroup isomorphic to $C_q: C_d$, where $q=\frac{r^d-\eta}{(r-\eta)(r-\eta,d)}$. Hence this maximal subgroup is of order product of two primes, if  both $d$ and $q$ are primes. Therefor $q$ is the largest primitive prime divisor of $r^d-\eta$, also $q\equiv 1, \pmod{d}$. Those subgroups are of class $\C_3$ (the stabilizers of extension fields of  $F_q$ of prime index $d$) of  Aschbacher's  collection~\cite[Table 1.2.A]{KlLi} or~\cite[Table 2.1]{BHD}.

The groups listed in Table~\ref{T-2} are the only simple groups with a maximal subgroup $M$ whose order is a product of two prime numbers. These groups could be potential candidates for which $M = \Sol_G(x)$ for some element $x \in G$.
{\Small
\begin{table}[h]
\begin{tabular}{|c|c|c|l|}
\hline
Type of Group	& Structure	& Maximal subgroup& Conditions \\
\hline
 &  $\PSL(2, 2^n)$,  & $D_{2q}$ & $q=2^n+1$ is a Fermat prime  \\ 
\cline{3-4}
 &$n\geq 2$ & $D_{2q}$ & $q=2^n-1$ is a Mersenne prime\\ 
\cline{2-4}
Projective Spesial & $ \PSL(2,r)$ ,& $ C_p : C_t $ & $n=1$ and $t=(p-1)/2$ is prime \\ 
\cline{3-4}
Linear Groups & $ r=p^n\geq 5$, & $ D_{2q}$ & $r\neq 7,9$ and  $q=(r+1)/2$ is prime\\ 
\cline{3-4}
 & $r$ is odd &$D_{2q}$  &  $r\geq 13$ and  $q=(r-1)/2$ is  prime  \\ 
\cline{2-4}
 & $\begin{array}{c} \PSL(d,r),\, d\geq 3\\ r=p^n\end{array}$ & $C_q : C_d$  & $\begin{array}{c} d\, \&\, q=\frac{r^d-1}{(r-1)(r-1,d)}\, \textrm{are primes}\\ q\equiv 1 \pmod{d} \end{array}$ \\  
\hline
$ \begin{array}{c} \text{Projective Spesial}\\\text{Unitary Groups } \end{array}$ & $\begin{array}{c} \PSU(d,r),\, d\geq 3\\  r=p^n\\ (d,r)\neq (3,5) \end{array}$ & $C_q : C_d$ &  $\begin{array}{c} d\, \&\, q=\frac{r^d+1}{(r+1)(r+1,d)}\, \textrm{are primes}\\
q\equiv 1 \pmod{d} \end{array}$ \\ 
\hline
Suzuki Groups & $Sz(2^{2n+1})$, $n>1$ & $D_{2q}$   & $q=2^{2n+1}-1$ is a Mersenne prime \\ 
\hline
Mathieu Group & $M_{23}$ & $ C_{23} : C_{11}$ & $q=23$, $p=11$ \\ 
\hline
Baby Monster & $\mathscr{B}=F_2+$ & $C_{47}:C_{23}$ &  $q=47$, $p=23$ \\ 
\hline
Monster & $\mathscr{M}=F_1$ & $C_{59}:C_{29}$ &  $q=59$, $p=29$ \\ 
\hline
\end{tabular}
\caption{Simple groups with local maximal subgroups  whose  orders are the product of two prime numbers}\label{T-2}
\end{table}}

Before discussing this section's main theorem, we need to establish two lemmas. These lemmas will specify the conditions under which a maximal subgroup can be equal to the solubilizer of certain elements of $G$.

\begin{lemma}\label{L-4.2}
Assume that for some group $G$ and $x\in G$, $|\Sol_G(x)|=pq$, where $p<q$ are primes. If $H=\N_G(\cyc{x})$ is a maximal subgroup of $G$, then $H^G$ is only conjugacy class of maximal subgroups of $G$ such that $q\mid|H|$.
\end{lemma}
\begin{proof}
 By Theorem~\ref{pq}, $|x|=q$, $p\mid q-1$ and  $G$ is simple. Also $\Sol_G(x)=\N_G(\cyc{x})$, where $\cyc{x}\in\Syl_q(G)$. Set $H=\cyc{x,y}=\Sol_G(x)$, where $|y|=p$.\\ Let $M$ be a maximal subgroup of $G$ such that $x\in M$. If $y\not\in M$, then $|\Sol_M(x)|=q$ hence $M$ is soluble and so we have contradiction $M=\cyc{x}$, by Lemma~\ref{sol}, (i) and (v). Therefore $y\in M$ and $H=M$ and so $H$ is only maximal subgroup of $G$ such that $x\in H$.

If $N$ is a maximal subgroup of $G$ such that $x^g\in N$ for some $g\in G$, then similar to previous paragraph  $y^g\in N$ and $N=H^g$, since $\Sol_G(x^g)=\Sol_G(x)^g$. Therefore $G$ contains only one conjugacy class $H^G$ of maximal subgroups such that $q\mid|H|$.
\end{proof}

The converse of Lemma~\ref{L-4.1} is also true and can be stated as follows:
\begin{lemma}\label{L-4.3}
Let $G$ be a group and $Q\in\Syl_q(G)$ be of prime order $q$. Assume that  $H=\N_G(Q)$ is maximal subgroup of $G$ and $H^G$ is only conjugacy class of maximal subgroups of $G$ such that $q\mid|H|$. Then for some $x\in G$ of order $q$, $\Sol_G(x)=\N_G(\cyc{x})$ is a maximal subgroup of $G$.
\end{lemma}
\begin{proof}
 Let $x\in H$ be of order $q$ and $z\in\Sol_G(x)$. Now for some maximal subgroup $M$ of $G$, $\cyc{x,z}\sub M$. Then $q\mid|M|$ and so $M\in H^G$. Thus for some $g\in G$, $M=H^g$ and  $x\in\N_G(\cyc{x}^g)$. Hence $\cyc{x}=\cyc{x}^g$ and $g\in\N_G(\cyc{x})$. Therefore $z\in M=H$ and so $\Sol_G(x)=H$
\end{proof}

\begin{theorem}\label{T-4.4}
Let $G$ be a finite group and $x\in G$ an element such that $|\Sol_G(x)|=pq$ for some primes $p<q$. Then $G$ is simple group, $|x|=q$, $p\mid q-1$. In additional, if $\Sol_G(x)$ is a maximal subgroup in $G$, then:
\begin{itemize}
\item[(i)]  $G\cong \mathscr{M}$, the monster group and $\Sol_G(x)\cong C_{59}: C_{29}$ for some $59$-element $x$.
\item[(ii)] $G\cong \mathscr{B}$, the baby monster group and $\Sol_G(x)\cong C_{47}: C_{23}$ for some $47$-element $x$.
\item[(iii)] $G\cong M_{23}$, the Mathieu group and $\Sol_G(x)\cong C_{23}: C_{11}$ for some $23$-element $x$.
\item[(iv)] $G\cong \PSL(2,2^n)$, where $q=2^n+1$ is a Fermat prime number and $\Sol_G(x)\cong D_{2q}$ for some $q$-element $x$.
\item[(v)] $G\cong\PSL(2,p)$, where $p\geq 5$ is prime and $\Sol_G(x)\cong C_p: C_t$ for some $p$-element $x$ if $t=(p-1)/2$ is odd prime and $\Sol_G(x)\cong D_{2q}$ where $q=(p+1)/2\geq 7$ is prime.
\item[(vi)] $G\cong \PSL(2,p^n)$, where $q=(p^n+1)/2\geq 7$ is a prime number and $\Sol_G(x)\cong D_{2q}$ for some $q$-element $x$.
\item[(vii)] $G\cong\PSL(d,p^n)$, where $d\geq 3$, $p$ and $q=\frac{r^d-1}{(r-1)(r-1,d)}$ are primes,  $r=p^n$. In this case $\Sol_G(x)\cong C_q : C_d$, for some $q$-element $x$.
\item[(viii)] $G\cong\PSU(d,p^n)$, $(d,p^n)\neq(3,5)$, $d\geq 3$, $p$ and $q=\frac{r^d+1}{(r+1)(r+1,d)}$ are primes, where $r=p^n$. In this case $\Sol_G(x)\cong C_q : C_d$, for some $q$-element $x$.
\end{itemize}
\end{theorem}

\begin{proof}
The first part of Theorem~\ref{T-4.4} is proved in Theorem~\ref{pq}. Therefore, we assume that $\Sol_G(x)$ is a maximal subgroup of $G$. The simple groups with a maximal subgroup of order $pq$ are listed in Table~\ref{T-2}. We need to identify which groups in Table~\ref{T-2} satisfy the theorem conditions.

Let, for some $q$-element $x$, $\N_G(\cyc{x})$ be a maximal subgroup of $G$ of order $pq$. Assume that $H$ is another maximal subgroup of $G$ such that $q\mid |H|$. Then for some $g\in G$, $x^g\in H$. If $H$ is not conjugate to $\N_G(\cyc{x})$, then $\N_H(\cyc{x^g})$ is a proper subgroup of $\N_G(\cyc{x})^g$,  so $|\N_H(\cyc{x^g})|= q$ and $H$ is soluble by Lemma~\ref{self}.

 By GAP~\cite{GAP}, we can compute the conjugacy class of maximal subgroups of the Mathieu group $M_{23}$ and the baby monster group $\mathscr{B}$. The group $M_{23}$ just contain one conjugacy class of maximal subgroups isomorphic to $C_{23}:C_{11}$ which is normalizer of its Sylow $23$-subgroup. So, by Lemma~\ref{L-4.3}, for any $23$-element $x$ of $M_{23}$, $\Sol_{M_{23}}(x)=C_{23}: C_{11}$.

Similarly, the baby  monster group $\mathscr{B}$, has only one conjugacy class of maximal subgroups isomorphic to $C_{47}:C_{23}$ which is normalizer of its Sylow $47$-subgroups. So, by Lemma~\ref{L-4.3}, for any $47$-element $x$ of $\mathscr{B}$, $\Sol_{\mathscr{B}}(x)=C_{47}: C_{23}$.

The Monster group $\mathscr{M}$ contains a maximal subgroup that is isomorphic to $C_{59}:C_{29}$. According to~\cite{Ivanov, MeSh1}, $\mathscr{M}$ contains no $2$-local soluble maximal subgroups. Furthermore, as noted in~\cite[p. 234]{Atlas} and \cite[Table 5.6]{Wilson}, the Monster group $\mathscr{M}$ has only one soluble maximal subgroup whose order is divisible by $59$. We note that another maximal subgroup $H$ of $\mathscr{M}$ which is not displayed in~\cite[Table 5.6]{Wilson}, satisfies $S\sub H\sub \Aut(S)$ where $S$ is a simple group. Hence, $H$ is not soluble (see~\cite[p. 258]{Wilson}). So by Lemma~\ref{L-4.3}, for some $59$-element $x$, $\Sol_G(x)\cong C_{59}: C_{29}$.

The Suzuki group has a maximal subgroup isomorphic to $D_{2q}$ of  order $2q$.  By Theorem~\ref{SZ},  the Suzuki group has another soluble maximal subgroup contains a Sylow $q$-subgroup which is not conjugate whit $D_{2q}$, so $\Sol_G(x)\neq\N_G(\cyc{x})$, where $x$ is an $q$-element. Therefore $|\Sol_G(x)|>q^2$ by Lemma~\ref{exp}. Consequently, Suzuki groups have been removed from our list of candidates.
 
Let $G\cong\PSL(2,2^n)$, where $q=2^n-1$ is a Mersenne prime number. Then $G$ has two class of maximal subgroups $D_{2q}$ and $E_{2^n} : C_q$ which orders divides by $q$. Therefore $\Sol_G(x)\neq\N_G(\cyc{x})$, where $x$ is a $q$-element (by Lemma~\ref{L-4.2}). Assume that $G\cong\PSL(2,2^n)$, where $q=2^n+1$ is a Fermat prime number. In this case $G$ only contains one conjugacy class of maximal subgroups $D_{2q}$. Therefore, for any $q$-element $x$, $\Sol_G(x)\cong D_{2q}$ (by Lemma~\ref{L-4.3}.

Assume that $G\cong\PSL(2,p)$, where $p\geq 5$ is odd prime. If $(p-1)/2$ is prime, then $G$ contains only one conjugacy class of maximal subgroup isomorphic to $C_p: C_{(p-1)/2}$ which is normalizer of some $p$-element. So by Lemma~\ref{L-4.3}, $\Sol_G(x)\cong C_p: C_{(p-1)/2}$. Assume that $(p+1)/2$ is a prime number. Then for $p\neq 7$, $G$ has a maximal subgroup of type $D_{p+1}$. If $p=5$, then $G\cong A_5$ has two conjugacy classes of maximal subgroups $D_6$ and $A_4$,  so this case is not what we want (by Lemma~\ref{L-4.2}). Therefore $p\geq 13$ and $G$ only contains one conjugacy class of maximal subgroups of type $D_{2q}$, where $q=(p+1)/2$ is prime.  By Lemma~\ref{L-4.3}, $\Sol_G(x)\cong D_{2q}$ for some $q$-element $x$.

Now assume that $G\cong\PSL(2,p^n)$, where $p$ is odd prime and $n\geq 2$. By using of $\GAP$ we see that $\PSL(2,9)$ is not what we want. First assume that  $q=(p^n-1)/2$ is a prime. Since in this case, $p^n\geq 13$, so $G$ has two classes of maximal subgroups $E_{p^n}: C_q$ and $D_{2q}$, that is not what we want (by Lemma~\ref{L-4.2}).
Let  $q=(p^n+1)/2$ be a prime when $p^n\neq 9$. Then $G$ has only one conjugacy class of maximal subgroups with dihedral structure $D_{2q}$ which is normalizer of its Sylow $q$-subgroup. Therefore $\Sol_G(x)\cong D_{2q}$, where $q=(p^n+1)/2$ is prime (by Lemma~\ref{L-4.3}).

Let $d\geq 3$. Then by Table~\ref{T-2}, $\PSL^{\eta}(d,r)$ has a maximal subgroup isomorphic to $C_q: C_d$, where $d$ is prime and  $q=\frac{r^d-\eta}{(r-\eta)(r-\eta,d)}$ is the largest primitive prime divisor of $r^d-\eta$, also $q\equiv 1, \pmod{d}$. Since the maximal subgroups of other classes, are subgroups of  $\PGL^{\eta}(d', r')$, where $d' < d$ or $r' < r$ ( see \cite[Tables, 3.5.A-B, Pages 70-71]{KlLi}), it follows that $q$ does not divide the order of those maximal subgroups (by Lemma~\ref{Primitive}-(ii)). Therefore, the group $G$ contains one conjugacy class of maximal subgroups that are isomorphic to $C_q : C_d$, which are the normalizer of some $q$-elements. Therefore $\Sol_G(x)\cong C_q : C_d$ for some $q$-element $x$ (by Lemma~\ref{L-4.3}).
\end{proof}
\subsection*{Funding } This paper is supported by the Research Affairs Office No. 4373 at the university of Tabriz.


\begin{thebibliography}{99}
\bibitem{Akbari1}
B. Akbari,, J. Chuharski, V. Sharan and Z. Slonim, Characterization of solubilizers of elements in minimal simple groups, \emph{Comm. Alg.} (2024). \href{https://doi.org/10.1080/00927872.2024.2428320}{https://doi.org/10.1080/00927872.2024.2428320}

\bibitem{Akbari}
B. Akbari, M. L. Lewis, J. Mirzajani and A. R. Moghaddamfar, The solubility graph associated with a finite group, \emph{Internat. J. Algebra Comput.}, 30 (8)(2020), 1555-1564.\\  \href{https://doi.org/10.1142/S0218196720500538}{https://doi.org/10.1142/S0218196720500538}

\bibitem{BHD}
J.N. Bray, D.F. Holt and C.M. Roney-Dougal, \emph{The maximal subgroups of the low-dimensional finite classical groups}, Cambridge University Press (2013).

\bibitem{CLSS}
A.M. Cohen, M.W. Liebeck, J. Saxl, and G.M. Seitz, The local maximal subgroups of exceptional groups of Lie type, finite and algebraic, \emph{Proc. London Math. Soc.}, ({\bf 3}), 64 (1992) 21-48. \href{https://doi.org/10.1007/BF00147353}{https://doi.org/10.1007/BF00147353}

\bibitem{Atlas}
J. H. Conway, R. T. Curtis, S. P. Norton, R A. Parker, and R A. Wilson, \emph{Atlas of finite groups}, Oxford University Press, New York (1985).

\bibitem{GAP}
GAP-\emph{ Groups, Algorithms and Programming},  Version 4.12.2; (2022). \href{https://www.gap-system.org}{https://www.gap-system.org}

\bibitem{GKPS}
R. Guralnick,  B. Kunyavski, E. Plotkin and A. Shalev, Thompson-like characterizations of the solvable radical, \emph{J. Algebra}, 300 (1)(2006), 365-375. \href{https://doi.org/10.1016/j.jalgebra.2006.03.001}{https://doi.org/10.1016/j.jalgebra.2006.03.001}

\bibitem{Hai}
D. Hai-Reuven, Nonsolvable graph of a finite group and solvabilizers, Preprint, \emph{arXiv: 1307.2924 v1 [math. GR]}, (2013).
\href{http://arxiv-export-lb.library.cornell.edu/abs/1307.2924}{http://arxiv-export-lb.library.cornell.edu/abs/1307.2924}

\bibitem{Ivanov}
Alexander A. Ivanov, \emph{Algebraic Combinatorics and the Monster Group}, LMS-Lecture Note Series, Cambridge University Press; (2023) pp. 200-245. \href{https://doi.org/10.1017/9781009338073.005}{https://doi.org/10.1017/9781009338073.005}

\bibitem{KlLi} P. Kleidman and M. Liebeck,
\emph{The subgroup structure of the finite classical groups}, Cambridge University Press (1990).

\bibitem{LPS}
M.W. Liebeck, C.E. Praege and  J. Saxl, A classification of the maximal subgroups of the finite alternating and symmetric groups, \emph{J. Algebra}, {\bf 111} (1987) 365-383. \href{https://doi.org/10.1016/0021-8693(87)90223-7}{https://doi.org/10.1016/0021-8693(87)90223-7}

\bibitem{LSS}
 M.W. Liebeck, J. Saxl, and G.M. Seitz, Subgroups of maximal rank in groups of Lie type, \emph{Proc. London Math. Soc.} (3) 65 (1992) 297-325. \href{ https://doi.org/10.1112/plms/s3-65.2.297}{ https://doi.org/10.1112/plms/s3-65.2.297}

\bibitem{MeSh1}
U. Meierfrankenfeld and S. Shpectorov, The maximal $2$-local subgroups of the Monster and Baby Monster II, (2023).
\href{https://users.math.msu.edu/users/meierfra/Preprints/2monster/2MNC.pdf}{https://users.math.msu.edu/users/meierfra/Preprints/2monster/2MNC.pdf}

\bibitem{Mousavi}
H. Mousavi, M. Poozesh and Y. Zamani, The impact of the solubilizer of an element on the structure of a finite group, Ricerche di Matematica, \href{https://doi.org/10.1007/s11587-023-00817-6}{https://doi.org/10.1007/s11587-023-00817-6}

\bibitem{Stw} 
W. B. Stewart, Groups having strongly self-centralizing $3$-centralizers, \emph{Proc. London Math. Soc.} (3) 26 (1973) 653-680. \href{https://doi.org/10.1112/plms/s3-26.4.653}{https://doi.org/10.1112/plms/s3-26.4.653}

\bibitem{Suzuki}
M. Suzuki, \emph{Group Theory I}. Springer, Berlin (1982).

\bibitem{Thompson}
J. G. Thompson, Nonsolvable finite groups all of whose local subgroups are solvable, \emph{Bull. Amer. Math. Soc.}, 74 (1968), 383–437. \href{https://www.ams.org/journals/bull/1968-74-03/S0002-9904-1968-11953-6/S0002-9904-1968-11953-6.pdf}{https://www.ams.org/journals/bull/1968-74-03/S0002-9904-1968-11953-6/S0002-9904-1968-11953-6.pdf}

\bibitem{Wilson}
Robert A. Wilson, \emph{The finite simple groups},  Springer-Verlag London Limited (2009).

\bibitem{Zsi}  K. Zsigmondy,  Zur Theorie der Potenzreste, \emph{Monatsh. f\"ur Math. u. Phys.}, V. 3 (1892), 265-284. \href{https://doi.org/10.1007/BF01692444}{https://doi.org/10.1007/BF01692444}
\end{thebibliography}
\end{document}